\documentclass{amsart}

\usepackage{amssymb, amsthm, amsmath, amsrefs, thmtools, enumerate}
\usepackage[all]{xy}
\usepackage[mathscr]{eucal}

\numberwithin{equation}{section}

\declaretheorem[name=Theorem, sibling=equation]{theorem}
\declaretheorem[name=Lemma, sibling=equation]{lemma}
\declaretheorem[name=Proposition, sibling=equation]{proposition}

\declaretheorem[name=Definition, style=definition, sibling=equation]{definition}
\declaretheorem[name=Notation, style=definition, sibling=equation]{notation}
\declaretheorem[name=Remark, style=definition, sibling=equation]{remark}

\DeclareMathOperator{\reg}{reg}
\DeclareMathOperator{\Mod}{-Mod}

\DeclareMathOperator{\HH}{H}
\DeclareMathOperator{\Ob}{Ob}

\newcommand{\N}{{\mathbb{N}}}
\newcommand{\Z}{{\mathbb{Z}}}
\newcommand{\F}{{\mathbb{F}}}
\newcommand{\hor}{{\mathrm{hor}}}
\newcommand{\ver}{{\mathrm{ver}}}
\newcommand{\VI}{{\mathrm{VI}}}
\newcommand{\FI}{{\mathrm{FI}}}
\newcommand{\ZZ}{{\mathbb{Z}_{\geqslant -1}}}
\newcommand{\SSigma}{{\overline{\Sigma}}}

\title{Bounding regularity of $\VI^m$-modules}

\author{Wee Liang Gan}

\author{Khoa Ta}

\begin{document}

\begin{abstract}
   Fix a finite field $\F$. Let $\VI$ be a skeleton of the category of finite dimensional $\F$-vector spaces and injective $\F$-linear maps. We study $\VI^m$-modules over a noetherian commutative ring in the nondescribing characteristic case. We prove that if a finitely generated $\VI^m$-module is generated in degree $\leqslant d$ and related in degree $\leqslant r$, then its regularity is bounded above by a function of $m$, $d$, and $r$. A key ingredient of the proof is a shift theorem for finitely generated $\VI^m$-modules.
\end{abstract}

\maketitle

\section*{Introduction} 

Fix a finite field $\F$ of order $q$. Fix also a noetherian commutative ring $\Bbbk$ such that $q$ is invertible in $\Bbbk$. 

Let $\N=\{0, 1, 2, \ldots\}$. For each $n\in \N$, let $\F^n$ be the $\F$-vector space of column vectors with $n$ entries, all entries lying in $\F$. In particular, $\F^0$ is the zero vector space. 

Let $\VI$ be the category defined as follows:
\begin{itemize}
   \item the set of objects of $\VI$ is $\N$;
   \item for all $a, b\in \N$, the morphisms in $\VI$ from $a$ to $b$ are the injective $\F$-linear maps from $\F^a$ to $\F^b$;
   \item composition of morphisms is composition of the maps.
\end{itemize} 
Let $m$ be an integer $\geqslant 1$ and $\VI^m$ the product of $m$ copies of $\VI$. A \emph{$\VI^m$-module} is a functor from $\VI^m$ to the category of $\Bbbk$-modules. A homomorphism of $\VI^m$-modules is a natural transformation of functors. Sam and Snowden have proven that every finitely generated $\VI^m$-module is noetherian (see Theorem 1.1.3, Proposition 4.3.5, and Theorem 8.3.1 in \cite{sams}). Thus the category of finitely generated $\VI^m$-modules is an abelian category. 

The main result of our present article says that if a finitely generated $\VI^m$-module is generated in degree $\leqslant d$ and related in degree $\leqslant r$, then its (Castelnuovo-Mumford) regularity is bounded above by a function $\rho_m(d,r)$. The category $\VI$ is an analog of a skeleton of the category $\FI$ of finite sets and injective maps. In our previous article \cite{gta}, we proved that the same function $\rho_m(d,r)$ gives an upper bound on the regularity of any $\FI^m$-module over an arbitrary commutative ring if it is generated in degree $\leqslant d$ and related in degree $\leqslant r$. While we will use some arguments from \cite{gta} in this article, the method of proof in \cite{gta} for $\FI^m$-modules does not work for $\VI^m$-modules; we discuss where it fails for $\VI^m$-modules in Remark \ref{rem: pf in gta}. 

Recall that in the $m=1$ case, Nagpal \cite{na}*{Theorem 5.11} proved that finitely generated $\VI$-modules have finite regularity. Subsequently, Gan and Li \cite{gl_vi}*{Theorem 3.2} proved that if a finitely generated $\VI$-module is generated in degree $\leqslant d$ and related in degree $\leqslant r$, then its regularity is bounded above by $\max\{ d, d+r-1 \}$, which is the same bound for $\FI$-modules proved by Church and Ellenberg \cite{ce}*{Theorem A}. The proof of the result of Gan and Li \cite{gl_vi}*{Theorem 3.2} used the shift theorem for finitely generated $\VI$-modules due to Nagpal \cite{na}*{Theorem 4.38}. To prove our main result, we shall use a shift theorem for finitely generated $\VI^m$-modules.

This article is organized as follows. In \S\ref{sec:prelim}, we recall basic definitions and two spectral sequences constructed in \cite{gta}; we also prove some preliminary results.
In \S\ref{sec: shift functors}, we study the natural shift functors $\Sigma_i$ and the modified shift functors $\SSigma_i$ on the category of $\VI^m$-modules. 
In \S\ref{sec: the functors Ki and Di}, we study the functors $K_i$ and $D_i$; the functor $K_i$ is the kernel of a natural morphism $\mathrm{id} \to \Sigma_i$, while the functor $D_i$ is the cokernel of a natural morphism $\mathrm{id} \to \SSigma_i$.
In \S\ref{sec: shift theorem}, we prove a shift theorem (Theorem \ref{thm: shift}) for finitely generated $\VI^m$-modules.
In \S\ref{sec: upper bound}, we prove our main result (Theorem \ref{thm: regularity}) giving an upper bound on the regularity of a finitely generated $\VI^m$-modules. 

\begin{notation}
   \begin{enumerate}
      \item We denote by $\ZZ$ the set $\{-1\}\cup\N$. 

      \item For any $\Bbbk$-module $A$ and group $G$ acting $\Bbbk$-linearly on $A$, we write $A_G$ for the quotient of $A$ by the $\Bbbk$-submodule generated by all elements of the form $g(a) - a$ where $g\in G$ and $a\in A$. 
      
      \item For any $n\in \N$, we write $I_n$ for the identity map on $\F^n$.  
   \end{enumerate}

\end{notation}

\section{Preliminaries} \label{sec:prelim}

\subsection{Principal projective modules}
The set of objects of $\VI^m$ is $\N^m$. For any $\mathbf{a}, \mathbf{b}\in \N^m$, we write $\VI^m(\mathbf{a}, \mathbf{b})$ for the set of morphisms in $\VI^m$ from $\mathbf{a}$ to $\mathbf{b}$. 

Let $V$ be any $\VI^m$-module. For each $\mathbf{a}\in \N^m$, we write $V_{\mathbf{a}}$ for $V(\mathbf{a})$. For each $f\in \VI^m(\mathbf{a},\mathbf{b})$, we write $f_*$ for the map $V(f): V_{\mathbf{a}} \to V_{\mathbf{b}}$. 

For each $\mathbf{n}\in \N^m$, let $M^{\VI^m}(\mathbf{n})$ be the $\VI^m$-module defined as follows:
\begin{itemize}
   \item for each $\mathbf{a}\in \N^m$, let $M^{\VI^m}(\mathbf{n})_{\mathbf{a}}$ be the free $\Bbbk$-module with basis $\VI^m(\mathbf{n}, \mathbf{a})$;
   \item for each $g\in \VI^m(\mathbf{a}, \mathbf{b})$, let $g_*: M^{\VI^m}(\mathbf{n})_{\mathbf{a}} \to M^{\VI^m}(\mathbf{n})_{\mathbf{b}}$ be the $\Bbbk$-linear map sending each $f\in \VI^m(\mathbf{n}, \mathbf{a})$ to the composition $gf\in \VI^m(\mathbf{n},\mathbf{b})$.
\end{itemize}
Then $M^{\VI^m}(\mathbf{n})$ is a projective $\VI^m$-module called a \emph{principal projective $\VI^m$-module}. We say that a $\VI^m$-module is \emph{free} if it is isomorphic to a direct sum of principal projective $\VI^m$-modules. A $\VI^m$-module is finitely generated if and only if it is a homomorphic image of a finite direct sum of principal projective $\VI^m$-modules.

Let $V^{(1)}, V^{(2)}, \ldots, V^{(m)}$ be any $\VI$-modules. The external tensor product 
\[ V^{(1)}\boxtimes V^{(2)}\boxtimes \cdots\boxtimes V^{(m)} \] 
is a $\VI^m$-module with
\[ \left(V^{(1)}\boxtimes V^{(2)}\boxtimes \cdots\boxtimes V^{(m)}\right)_{(a_1, a_2, \ldots, a_m)} = V^{(1)}_{a_1}\otimes_{\Bbbk} V^{(2)}_{a_2}\otimes_{\Bbbk} \cdots\otimes_{\Bbbk} V^{(m)}_{a_m} \] 
for all $(a_1, a_2, \ldots, a_m)\in \N^m$. For any morphism $f=(f_1, f_2, \ldots, f_m)$ in $\VI^m(\mathbf{a}, \mathbf{b})$, 
\[ f_*:  \left(V^{(1)}\boxtimes V^{(2)}\boxtimes \cdots\boxtimes V^{(m)}\right)_{\mathbf{a}} \to  \left(V^{(1)}\boxtimes V^{(2)}\boxtimes \cdots\boxtimes V^{(m)}\right)_{\mathbf{b}} \]
is the map $(f_1)_* \otimes (f_2)_* \otimes \cdots\otimes (f_m)_*$.  

For each $\mathbf{n}=(n_1, n_2, \ldots, n_m)\in \N^m$, we have a canonical isomorphism:
\begin{equation} \label{eq: tensor}
   M^{\VI^m}(\mathbf{n}) \cong M^{\VI}(n_1) \boxtimes M^{\VI}(n_2) \boxtimes \cdots \boxtimes M^{\VI}(n_m).
\end{equation}

\subsection{Homology}
We write $\VI^m\Mod$ for the category of $\VI^m$-modules. 

Let $V$ be any $\VI^m$-module. For each $\mathbf{n}\in \N^m$, let $\widetilde{V}_{\mathbf{n}}$ be the $\Bbbk$-submodule of $V_{\mathbf{n}}$ generated by 
\[ \left\{ f_*(v) \ \Big|\ \mathbf{a}\in \N^m,\ \mathbf{a}\neq \mathbf{n},\ f\in \VI^m(\mathbf{a}, \mathbf{n}),\ v\in V_{\mathbf{a}} \right\}. \]
The assignment $\mathbf{n}\mapsto \widetilde{V}_{\mathbf{n}}$ defines a $\VI^m$-submodule $\widetilde{V}$ of $V$. Now define a functor
\[ \HH^{\VI^m}_0 : \VI^m\Mod \to \VI^m\Mod \] 
by $ \HH^{\VI^m}_0(V)=V/\widetilde{V}$. Then $\HH^{\VI^m}_0$ is a right exact functor. For each integer $i\geqslant 1$, let 
\[ \HH^{\VI^m}_i : \VI^m\Mod \to \VI^m\Mod \] 
be the $i$-th left derived functor of $\HH^{\VI^m}_0$.

We say that a $\VI^m$-module $V$ is \emph{homology acyclic} if $\HH^{\VI^m}_i(V)=0$ for all $i\geqslant 1$. For example, free $\VI^m$-modules are homology acyclic.

\subsection{Regularity}
For each $\mathbf{n}=(n_1, \ldots, n_m)\in \N^m$, set 
\[ |\mathbf{n}| = n_1 + \cdots + n_m \in \N. \]

Let $V$ be any $\VI^m$-module. The \emph{degree} $\deg V$ of $V$ is defined by 
\[ \deg V = \begin{cases} 
               \sup \{ |\mathbf{n}| \mid V_\mathbf{n} \neq 0 \} & \mbox{ if } V\neq 0,\\
               -1 & \mbox{ if } V=0.
            \end{cases} \]
We have $\deg V \in \ZZ \cup \{\infty\}$. We say that $V$ has \emph{finite degree} if $\deg V < \infty$. 

For each $i\in \N$, let 
\[ t_i(V) = \deg \HH^{\FI^m}_i(V). \] 

Let $d\in\ZZ$. We say that $V$ is \emph{generated in degree $\leqslant d$} if $t_0(V) \leqslant d$. Equivalently, $V$ is generated in degree $\leqslant d$ if and only if there exists an epimorphism $P \to V$ where $P$ is a free $\VI^m$-module of the form 
\[ P = \bigoplus_{j\in J} M^{\VI^m}(\mathbf{n}_j), \]
where $J$ is some indexing set and $|\mathbf{n}_j| \leqslant d$ for all $j\in J$.

Let $d,r\in\ZZ$. We say that $V$ is \emph{generated in degree $\leqslant d$ and related in degree $\leqslant r$} if there exists a short exact sequence
\[ 0 \to U \to P \to V \to 0 \]
where $P$ is a free $\VI^m$-module generated in degree $\leqslant d$, and $U$ is a $\VI^m$-module generated in degree $\leqslant r$. We can choose $P$ with $t_0(P)=t_0(V)$ and in this case, from the long exact sequence induced by the above short exact sequence, we have 
\[ t_1(V) \leqslant t_0(U) \leqslant \max\{ t_0(V), t_1(V) \}. \]
In particular, $t_1(V) \leqslant r$.  

The \emph{regularity} $\reg(V)$ of $V$ is defined by 
\[ \reg(V) = \sup \{ t_i(V)-i \mid i\geqslant 0\}. \]

\subsection{Spectral sequences}
Assume that $m\geqslant 2$ and consider $\VI^m$ as the product category $\VI\times \VI^{m-1}$. 

Let $n_1\in \N$ and $(n_2, \ldots, n_m)\in \N^{m-1}$. Then we have inclusion functors:
\begin{align*}   
   \VI &\to \VI^m, \quad a_1 \mapsto (a_1, n_2, \ldots, n_m);\\
   \VI^{m-1} &\to \VI^m, \quad (a_2, \ldots, a_m) \mapsto (n_1, a_2, \ldots, a_m).
\end{align*}
They induce, respectively, the restriction functors:
\begin{align*}   
   \VI^m\Mod &\to \VI\Mod, \quad V\mapsto V_{(-, n_2, \ldots, n_m)};\\
   \VI^m\Mod &\to \VI^{m-1}\Mod, \quad V\mapsto V_{(n_1, -)}.
\end{align*}

For each integer $i\geqslant 0$, the authors defined in \cite{gta} a pair of functors 
\begin{align*} 
   \HH^\hor_i &: \VI^m\Mod \to \VI^m\Mod,\\
   \HH^\ver_i &: \VI^m\Mod \to \VI^m\Mod
\end{align*}
called the \emph{$i$-th horizontal homology} of $V$ and the \emph{$i$-th vertical homology} of $V$, respectively.
Let us briefly recall the definition of the horizontal homology; the definition of the vertical homology is similar. 

\begin{notation} \label{horizontal submodule}
   For any $\VI^m$-module $V$ and $\mathbf{n}=(n_1, \ldots, n_m)\in \N^m$, let $V^\hor_{\mathbf{n}}$ be the $\Bbbk$-submodule of $V_{\mathbf{n}}$ generated by all elements of the form 
   \[ (f, I_{n_2}, \ldots, I_{n_m})_*(v) \]
   for some $a_1\in \N$, $a_1<n_1$, $f\in \VI(a_1, n_1)$, $v\in V_{(a_1, n_2, \ldots, n_m)}$.
\end{notation}

Let $V$ be any $\VI^m$-module. The assignment $\mathbf{n} \mapsto V^\hor_{\mathbf{n}}$ defines a $\VI^m$-submodule $V^\hor$ of $V$. Now define the functor
\[ \HH^\hor_0 : \VI^m\Mod \to \VI^m\Mod, \quad V\mapsto V/V^\hor. \] 
Then $\HH^\hor_0$ is a right exact functor. For each integer $i\geqslant 1$, let $\HH^\hor_i$ be the $i$-th left derived functor of $\HH^\hor_0$. 

\begin{lemma} \label{lem: hor prin proj res}
   Assume $m\geqslant 2$ and consider $\VI^m$ as the product category $\VI\times \VI^{m-1}$. Let $\mathbf{n}\in \N^m$ and $a_1\in \N$. 
   Then $(\HH^\hor_0(M^{\VI^m}(\mathbf{n})))_{(a_1, -)}$ is a free $\VI^{m-1}$-module generated in degree $\leqslant |\mathbf{n}|$. 
\end{lemma}

\begin{proof}
   Write $\mathbf{n}=(n_1, \ldots, n_m)$. 
   
   If $a_1\neq n_1$, then $(\HH^\hor_0(M^{\VI^m}(\mathbf{n})))_{(a_1, -)} = 0$.

   If $a_1=n_1$, then $(\HH^\hor_0(M^{\VI^m}(\mathbf{n})))_{(a_1, -)} = (M^{\VI^m}(\mathbf{n}))_{(n_1, -)}$, which by \cite{gta}*{Lemma 2.3} is isomorphic to 
   \[ \bigoplus_{f\in \VI(n_1, n_1)} M^{\VI^{m-1}}((n_2, \ldots, n_m)). \]
\end{proof}

By \cite{gta}*{Lemma 2.4}, we have:
   \begin{equation} \label{eq:hor and ver homology}
      \begin{aligned}
         (\HH^\hor_i (V))_{(-, n_2, \ldots, n_m)} &\cong \HH^{\VI}_i(V_{(-, n_2, \ldots, n_m)}),\\
         (\HH^\ver_i (V))_{(n_1, -)} &\cong \HH^{\VI^{m-1}}_i(V_{(n_1, -)}).
      \end{aligned}
   \end{equation}
Moreover, by \cite{gta}*{Theorem 2.5}, there are two convergent first-quadrant spectral sequences:
   \begin{equation} \label{eq:hor and ver spectral sequences}
      \begin{aligned}
         ^{I}\!E^2_{pq} = \HH^\ver_p (\HH^\hor_q (V)) &\Rightarrow \HH^{\VI^m}_{p+q}(V),\\
         ^{II}\!E^2_{pq} = \HH^\hor_p (\HH^\ver_q (V)) &\Rightarrow \HH^{\VI^m}_{p+q}(V).
      \end{aligned}
   \end{equation}

\begin{lemma} \label{lem: spec seq degenerate}
   Assume $m\geqslant 2$ and consider $\VI^m$ as the product category $\VI\times \VI^{m-1}$. Let $V$ be a $\VI^m$-module such that for each $(n_2, \ldots, n_m)\in \N^{m-1}$, the $\VI$-module $V_{(-, n_2, \ldots, n_m)}$ is homology acyclic. Then for all $p\geqslant 0$, 
   \[ \HH^{\VI^m}_p(V) \cong \HH_p^\ver (\HH_0^\hor(V)). \]
\end{lemma}

\begin{proof}
   By \eqref{eq:hor and ver homology}, we know that $\HH^\hor_q(V)=0$ for all $q\geqslant 1$. Thus in \eqref{eq:hor and ver spectral sequences}, we have $^{I}\!E^2_{pq}=0$ for all $q\geqslant 1$. The lemma follows.
\end{proof}

\begin{proposition} \label{prop: bound ti by 2i}
   Let $m, \alpha, \beta\in \mathbb{Z}$ with $m\geqslant 2$ and $\alpha, \beta\geqslant -1$. Let $V$ be a $\VI^m$-module such that for all $(n_1, n_2, \ldots, n_m)\in \N^m$:
   \begin{align*}
      \reg\left(V_{(-, n_2, \ldots, n_m)}\right) &\leqslant \alpha,\\
      \reg\left(V_{(n_1, -)}\right) &\leqslant \beta.
   \end{align*} 
   Then 
   \[ t_i(V) \leqslant \max\{-1, 2i + \alpha + \beta\} \quad \mbox{ for all }i\in \N. \]
\end{proposition}

We omit the proof of the above proposition since it is essentially the same as the proof of \cite{gta}*{Proposition 4.1}, with $\VI$ in place of $\FI$.

\section{Shift functors} \label{sec: shift functors}

\subsection{Natural shift functors} \label{subsec:shift}
We shall write linear maps from $\F^a$ to $\F^b$ ($a,b\in\N$) as $b\times a$ matrices; in particular, $I_a$ is the $a\times a$ identity matrix. 

Let $i\in \{1, 2, \ldots, m\}$. Define $\mathbf{e}_i\in \N^m$ by 
\[ \mathbf{e}_i = (0, \ldots, 1, \ldots, 0), \]
where $1$ is in the $i$-th coordinate. 

There is a functor $\iota_i : \VI^m \to \VI^m$ defined on objects by 
\begin{align*} 
   \iota_i : \Ob(\VI^m) &\to \Ob(\VI^m), \\ 
   \mathbf{a} &\mapsto \mathbf{a}+\mathbf{e}_i,
\end{align*}
and on morphisms by
\begin{align*}
    \iota_i : \VI^m(\mathbf{a}, \mathbf{b}) &\to \VI^m(\mathbf{a}+\mathbf{e}_i, \mathbf{b}+\mathbf{e}_i), \\
     (f_1, \ldots, f_m) &\mapsto (g_1, \ldots, g_m), 
\end{align*}
where $g_j=f_j$ for all $j\neq i$ and
\[ g_i = \left( \begin{array}{c|ccc} 
   1 & 0 & \cdots & 0 \\
   \hline
   0 &   &        &   \\
   \vdots & & f_i &   \\
   0 &   &        &  
   \end{array} \right). \]

The $i$-th \emph{natural shift functor} 
\[ \Sigma_i : \VI^m\Mod \to \VI^m\Mod \]
is defined to be the pullback via $\iota_i$. Thus, for any $\VI^m$-module $V$, we have 
\[ (\Sigma_i V)_{\mathbf{a}} = V_{\mathbf{a}+\mathbf{e}_i} \quad\mbox{ for all } \mathbf{a}\in \N^m. \]

\begin{notation} \label{nota:iota}
   When $m=1$, we write $\iota$ and $\Sigma$ for $\iota_1$ and $\Sigma_1$, respectively.
\end{notation}

For any $\VI$-modules $V^{(1)}, \ldots, V^{(m)}$, we have:
\begin{equation}\label{eq: tensor2}
\Sigma_i\left( V^{(1)}\boxtimes \cdots\boxtimes V^{(m)} \right) 
= V^{(1)}\boxtimes \cdots \boxtimes \left( \Sigma V^{(i)} \right) \boxtimes \cdots\boxtimes V^{(m)}. 
\end{equation}

\begin{proposition} \label{prop:shift_of_free}
   Let $\mathbf{n}=(n_1, \ldots, n_m)\in \N^m$. 
   
   If $n_i=0$, then 
      \[ \Sigma_i M^{\VI^m}(\mathbf{n}) \cong M^{\VI^m}(\mathbf{n}). \]
   
   If $n_i\geqslant 1$, then
      \[ \Sigma_i M^{\VI^m}(\mathbf{n}) \cong \left( M^{\VI^m}(\mathbf{n})^{\oplus k} \right) \oplus 
      \left( M^{\VI^m}(\mathbf{n}-\mathbf{e_i})^{\oplus \ell} \right) \]
   where 
   \[ k=q^{n_i}, \qquad \ell=(q^{n_i}-1)q^{{n_i}-1}. \]
\end{proposition}

\begin{proof}
   By \eqref{eq: tensor} and \eqref{eq: tensor2}, it suffices to prove the lemma when $m=1$. 
   
   It is plain that 
   \[ \Sigma M^{\VI}(0) \cong M^{\VI}(0). \]

   Let $n\in \N$ with $n\geqslant 1$. We need to prove that 
   \[    \Sigma M^{\VI}(n) \cong \left( M^{\VI}(n)^{\oplus q^n} \right) \oplus \left( M^{\VI}(n-1)^{\oplus (q^n-1)q^{n-1}} \right).   \]
   This is sketched in \cite{gl_coinduction}*{\S6.1}; for the convenience of the reader, we give the details below.

   Let $a\in \N$ with $a\geqslant n-1$. Let $f$ be any element of the basis $\VI(n,a+1)$ of $(\Sigma M^{\VI}(n))_a$. Then $f$ is a $(a+1)\times n$ matrix of rank $n$. Let $\beta_f$ be the first row of $f$ and let $\gamma_f$ be the $a\times n$ matrix formed by the last $a$ rows of $f$. The rank of $\gamma_f$ must be $n$ or $n-1$. If the rank of $\gamma_f$ is $n$, then $\ker \gamma_f = 0$.
   If the rank of $\gamma_f$ is $n-1$, then $\ker \gamma_f$ is a one-dimensional subspace of $\F^n$ and the restriction of the row vector $\beta_f: \F^n \to \F$ to $\ker \gamma_f$ is nonzero. 

   For each row vector $\beta: \F^n \to \F$, let $V(\beta)_a$ be the $\Bbbk$-submodule of $(\Sigma M^{\VI}(n))_a$ generated by all $f\in \VI(n, a+1)$ such that: 
   \begin{itemize}
      \item $\beta_f=\beta$,
      \item $\ker \gamma_f = 0$.
   \end{itemize}
   This defines a $\VI$-submodule $V(\beta)$ of $\Sigma M^{\VI}(n)$. It is easy to see that $V(\beta)$ is isomorphic to $M^{\VI}(n)$. There are $q^n$ choices of row vectors $\beta:\F^n \to \F$.

   For each one-dimensional subspace $B$ of $\F^n$ and row vector $\beta: \F^n\to \F$ whose restriction to $B$ is nonzero, let $W(B, \beta)_a$ be the $\Bbbk$-submodule of $(\Sigma M^{\VI}(n))_a$ generated by all $f\in \VI(n, a+1)$ such that:
   \begin{itemize}
      \item $\beta_f=\beta$,
      \item $\ker \gamma_f=B$.
   \end{itemize}
   This defines a $\VI$-submodule $W(B, \beta)$ of $\Sigma M^{\VI}(n)$. Observe that when $\ker \gamma_f=B$, the map $\gamma_f: \F^n\to \F^a$ factors uniquely through the quotient $\F^n/B$ (which is isomorphic to $\F^{n-1}$).
   From this, it is easy to see that $W(B, \beta)$ is isomorphic to $M^{\VI}(n-1)$. There are $(q^n-1)/(q-1)$ choices of one-dimensional subspaces $B$ of $\F^n$. For each such choice of $B$, there are $q^n-q^{n-1}$ choices of row vectors $\beta: \F^n \to \F$ whose restriction to $B$ is nonzero. Therefore there are $(q^n-1)q^{n-1}$ choices of such pairs $(B, \beta)$.

   We conclude from the above that $\Sigma M^{\VI}(n)$ is the direct sum of $q^n$ copies of $M^{\VI}(n)$ and $(q^n-1)q^{n-1}$ copies of $M^{\VI}(n-1)$.
\end{proof}

\begin{lemma} \label{lem:natural shifts commute}
   Let $i,j\in \{1,\ldots, m\}$. Then $\Sigma_i \Sigma_j = \Sigma_j \Sigma_i$.
\end{lemma}

\begin{proof}
   This follows from $\iota_i \iota_j = \iota_j \iota_i$.
\end{proof}

\subsection{Modified shift functor on $\VI\Mod$}
Following Nagpal \cite{na}, we shall define a functor 
\[ \overline{\Sigma}: \VI\Mod \to \VI\Mod \]
which turns out to be more useful than the natural shift functor. We call $\overline{\Sigma}$ the \emph{modified shift functor} on $\VI\Mod$. 

Let $a\in \N$. For each $c\in \F^a$, define $\sigma(c)\in \VI(a+1, a+1)$ by 
\[ \sigma(c) = \left( \begin{array}{c|ccc} 
   1 & 0 & \cdots & 0 \\
   \hline
    &   &        &   \\
   c & & I_a &   \\
    &   &        &  
   \end{array} \right). \]
Let 
\[ U(a)=\{\sigma(c) \mid c\in \F^a\}. \]  
Then $U(a)$ is a subgroup of $\VI(a+1, a+1)$. We have $\sigma(c) \sigma(c')=\sigma(c+c')$ for all $c,c'\in \F^a$.

Let $V$ be any $\VI$-module. The $\VI$-module $\SSigma V$ is defined as follows:
\begin{itemize}
   \item For each $a\in \N$, let
   \[ (\SSigma V)_a = (V_{a+1})_{U(a)}. \]
   \item For each $f\in \VI(a,b)$, let  
   \[ f_* : (\SSigma V)_a \to (\SSigma V)_b \]
   be the map sending the class represented by $v\in V_{a+1}$ to the class represented by $\iota(f)_*(v)$ (see Notation \ref{nota:iota}). The map $f_*$ is well-defined because for all $c\in \F^a$:   
   \[ \iota(f)\sigma(c) = \sigma(fc)\iota(f); \]
   hence for all $v\in V_{a+1}$:
   \[ \iota(f)_*(\sigma(c)_*(v) - v) =  \sigma(fc)_*(w) - w \quad \mbox{ where } w= \iota(f)_*(v). \]
\end{itemize} 

For any homomorphism $\eta:V\to W$ of $\VI$-modules, it is plain that $\Sigma\eta: \Sigma V \to \Sigma W$ descends to a homomorphism $\SSigma \eta: \SSigma V\to \SSigma W$.

Now let $n\in\N$ with $n\geqslant 1$. Let $a\in \N$ with $a\geqslant n-1$. For each $f\in \VI(n,a+1)$, let $\beta_f$ be the first row of $f$ and let $\gamma_f$ be the $a\times n$ matrix formed by the last $a$ rows of $f$. 

\begin{lemma} \label{lem:reduced_form}
   Let $f\in \VI(n,a+1)$. Suppose that $\beta_f=(\beta_1 \ \beta_2 \ \ldots \ \beta_n)$.
   \begin{enumerate}[(i)]
      \item If $\beta_f=0$, then for all $c\in \F^a$, one has $\sigma(c)f=f$.  
      \item If $\beta_f\neq 0$, then for each $j\in\{1,\ldots,n\}$ such that $\beta_j\neq 0$, there exists a unique $c\in \F^a$ such that the $j$-th column of $\gamma_{\sigma(c)f}$ is $0$.
   \end{enumerate}
\end{lemma}

\begin{proof}
Let $c\in \F^a$ and $g=\sigma(c)f$. Then
\[ \beta_g = \beta_f, \quad \gamma_g = c\beta_f + \gamma_f. \]
Thus if $\beta_f=0$, then $g=f$. This proves (i).

Suppose that $\beta_j\neq 0$. Let $u$ be the $j$-th column of $\gamma_f$ and $v$ be the $j$-th column of $\gamma_g$. Then 
\[ v = c\beta_j + u. \] 
Thus $v=0$ if and only if $c=-\beta_j^{-1}u$. This proves (ii).
\end{proof}

The preceding lemma allows us to make the following definition.

\begin{definition} \label{def:reduced_form}
   For any morphism $f\in \VI(n,a+1)$, define the morphism $\overline{f}\in \VI(n,a+1)$ as follows:
   \begin{itemize}
      \item if $\beta_f=0$, then $\overline{f}=f$;
      \item if $\beta_f=(\beta_1 \ \beta_2 \ \ldots \ \beta_n)$ with $\beta_1=\cdots=\beta_{j-1}=0$ and $\beta_j\neq 0$, then $\overline{f}=\sigma(c)f$ where $c$ is the unique vector in $\F^a$ such that the $j$-th column of $\gamma_{\sigma(c)f}$ is $0$.
   \end{itemize}    
\end{definition}

\begin{proposition} \label{prop:barshift1}
   Let $n\in \N$. 
   
   If $n=0$, then 
   \[ \SSigma M^{\VI}(n) \cong M^{\VI}(n). \]    
   
   If $n\geqslant 1$, then 
   \[ \SSigma M^{\VI}(n) \cong M^{\VI}(n) \oplus \left( M^{\VI}(n-1)^{\oplus (q^n-1)} \right). \]
\end{proposition}

\begin{proof}
   The $n=0$ case is trivial so assume that $n\geqslant 1$. 

   Let $a\in \N$ with $a\geqslant n-1$. For each row vector $\beta: \F^n \to \F$, let $P(\beta)_a$ be the $\Bbbk$-submodule of $(\Sigma M^{\VI}(n))_a$ generated by all $f\in \VI(n,a+1)$ such that $\beta_f=\beta$ and $\overline{f}=f$ (see Definition \ref{def:reduced_form}). In other words:
   \begin{itemize}
      \item if $\beta=0$, then $P(\beta)_a$ is generated by all $f\in \VI(n,a+1)$ such that $\beta_f=0$;
      \item if $\beta=(\beta_1 \ \beta_2 \ \ldots \ \beta_n)$ with $\beta_1=\cdots=\beta_{j-1}=0$ and $\beta_j\neq 0$, then $P(\beta)_a$ is generated by all $f\in \VI(n,a+1)$ such that $\beta_f=\beta$ and the $j$-th column of $\gamma_f$ is $0$.
   \end{itemize}    
   This defines a $\VI$-submodule $P(\beta)$ of $\Sigma M^{\VI}(n)$. It is easy to see that:
   \begin{itemize}
      \item if $\beta=0$, then $P(\beta)$ is isomorphic to $M^{\VI}(n)$;
      \item if $\beta\neq 0$, then $P(\beta)$ is isomorphic to $M^{\VI}(n-1)$.
   \end{itemize} 
   Let $P$ be the sum of $P(\beta)$ over all row vectors $\beta: \F^n \to \F$. Then 
   \[ P \cong M^{\VI}(n) \oplus \left( M^{\VI}(n-1)^{\oplus (q^n-1)} \right). \] 
   
   Let $\eta$ be the composition of the inclusion map $P\to \Sigma M(n)$ and the quotient map $\Sigma M(n) \to \SSigma M(n)$. It suffices to show that $\eta$ is an isomorphism. Let $a\in \N$ with $a\geqslant n-1$. We need to show that $\eta_a : P_a \to (\SSigma M(n))_a$ is bijective. 

   Let $\zeta_a: (\Sigma M(n))_a \to P_a$ be the $\Bbbk$-linear map sending each morphism $f\in \VI(n,a+1)$ to $\overline{f}$. Then $\zeta_a$ factors through $(\SSigma M(n))_a$ to give a $\Bbbk$-linear map $\overline{\zeta}_a: (\SSigma M(n))_a \to P_a$. It is plain that $\eta_a$ and $\overline{\zeta}_a$ are inverses of each other, hence $\eta_a$ is bijective.
\end{proof}

Proposition \ref{prop:barshift1} is essentially due to Nagpal; see \cite{na}*{Proposition 4.12} whose statement and proof are more abstract than the one we give above. 

\subsection{Modified shift functors on $\VI^m\Mod$}
Let $i\in \{1, 2, \ldots, m\}$. We now define the $i$-th \emph{modified shift functor} 
\[ \SSigma_i : \VI^m\Mod \to \VI^m\Mod. \]

Let $\mathbf{a} = (a_1, \ldots, a_m)\in \N^m$. For each $c\in \F^{a_i}$, let 
\[ \sigma_i(c) = (g_1, \ldots, g_m) \in \VI^m(\mathbf{a}+\mathbf{e}_i, \mathbf{a}+\mathbf{e}_i) \]
where $g_j = I_{a_j}$ for all $j\neq i$ and $g_i = \sigma(c)$. Let 
\[ U_i(\mathbf{a}) = \{ \sigma_i(c) \mid c\in \F^{a_i} \}. \] 
Then $U_i(\mathbf{a})$ is a subgroup of $\VI^m(\mathbf{a}+\mathbf{e}_i, \mathbf{a}+\mathbf{e}_i)$, and $U_i(\mathbf{a})\cong U(a_i)$.

Let $V$ be any $\VI^m$-module. The $\VI^m$-module $\SSigma_i V$ is defined as follows:
\begin{itemize}
   \item For each $\mathbf{a}\in \N^m$, let
   \[ (\SSigma_i V)_{\mathbf{a}} = (V_{\mathbf{a}+\mathbf{e}_i})_{U_i(\mathbf{a})}. \]

   \item For each $f\in \VI^m(\mathbf{a},\mathbf{b})$, let  
   \[ f_* : (\SSigma_i V)_{\mathbf{a}} \to (\SSigma_i V)_{\mathbf{b}} \]
   be the map sending the class represented by $v\in V_{\mathbf{a}+\mathbf{e}_i}$ to the class represented by $\iota_i(f)_*(v)$. The map $f_*$ is well-defined because for all $c\in \F^{a_i}$:   
   \[ \iota_i(f)\sigma_i(c) = \sigma_i(fc)\iota_i(f); \]
   hence for all $v\in V_{\mathbf{a}+\mathbf{e}_i}$:
   \[ \iota_i(f)_*(\sigma_i(c)_*(v) - v) =  \sigma_i(fc)_*(w) - w \quad \mbox{ where } w= \iota_i(f)_*(v). \]
\end{itemize} 

If $\eta: V\to W$ is a homomorphism of $\VI^m$-modules, then $\Sigma_i\eta: \Sigma_i V\to \Sigma_i W$ descends to a homomorphism $\SSigma_i\eta: \SSigma_i V\to \SSigma_i W$.

For any $\VI$-modules $V^{(1)}, \ldots, V^{(m)}$, we have:
\begin{equation}\label{eq: tensor_i}
\SSigma_i\left( V^{(1)}\boxtimes \cdots\boxtimes V^{(m)} \right) 
= V^{(1)}\boxtimes \cdots \boxtimes \left( \SSigma V^{(i)} \right) \boxtimes \cdots\boxtimes V^{(m)}. 
\end{equation}

\begin{proposition} \label{prop:barshift2}
   Let $\mathbf{n}=(n_1, \ldots, n_m)\in \N^m$. 

   If $n_i=0$, then 
   \[ \SSigma_i M^{\VI^m}(\mathbf{n}) \cong M^{\VI^m}(\mathbf{n}). \]

   If $n_i\geqslant 1$, then 
   \[ \SSigma_i M^{\VI^m}(\mathbf{n}) \cong M^{\VI^m}(\mathbf{n}) \oplus \left( M^{\VI^m}(\mathbf{n}-\mathbf{e}_i)^{\oplus (q^{n_i}-1)} \right). \]
\end{proposition}

\begin{proof}
   This is immediate from \eqref{eq: tensor}, \eqref{eq: tensor_i}, and Proposition \ref{prop:barshift1}.
\end{proof}

\begin{lemma} \label{lem:modified shifts commute}
   Let $i,j\in \{1,\ldots, m\}$. Then $\SSigma_i \SSigma_j \cong \SSigma_j \SSigma_i$.
\end{lemma}

\begin{proof}
   This is trivial if $i=j$ so assume that $i\neq j$. 
   
   Let $\mathbf{a}\in \N^m$. For any $\VI^m$-module $V$, we have
   \begin{align*}
      ( \SSigma_i \SSigma_j V )_{\mathbf{a}} 
      &= ((\SSigma_j V)_{\mathbf{a}+\mathbf{e}_i})_{U_i(\mathbf{a})} \\
      &= ((V_{\mathbf{a}+\mathbf{e}_i+\mathbf{e}_j})_{U_j(\mathbf{a}+\mathbf{e}_i)})_{U_i(\mathbf{a}+\mathbf{e}_j)},
   \end{align*}
   where we used the fact that $\iota_j(U_i(\mathbf{a})) = U_i(\mathbf{a}+\mathbf{e}_j)$. Thus 
   \[ ( \SSigma_i \SSigma_j V )_{\mathbf{a}} \cong V_{\mathbf{a}+\mathbf{e}_i+\mathbf{e}_j}/ (A+B) \]
   where:
   \begin{itemize}
      \item $A$ is the $\Bbbk$-submodule of $V_{\mathbf{a}+\mathbf{e}_i+\mathbf{e}_j}$ generated by all elements of the form $g_*(v)-v$ for some $g\in U_i(\mathbf{a}+\mathbf{e_j})$ and $v\in V_{\mathbf{a}+\mathbf{e}_i+\mathbf{e}_j}$;
      \item $B$ is the $\Bbbk$-submodule of $V_{\mathbf{a}+\mathbf{e}_i+\mathbf{e}_j}$ generated by all elements of the form $g_*(v)-v$ for some $g\in U_j(\mathbf{a}+\mathbf{e_i})$ and $v\in V_{\mathbf{a}+\mathbf{e}_i+\mathbf{e}_j}$.
   \end{itemize}   
   Similarly we also have 
   \[ (\SSigma_j \SSigma_i V )_{\mathbf{a}} \cong V_{\mathbf{a}+\mathbf{e}_i+\mathbf{e}_j}/(A+B). \]
\end{proof}

\begin{remark} \label{rem: exactness sigma}
   The functor $\Sigma_i$ is exact. Since $|U_i(\mathbf{a})|$ is invertible in $\Bbbk$ for all $\mathbf{a}\in \N^m$, the functor $\SSigma_i$ is exact. 
\end{remark}

\begin{lemma} \label{lem: shift preserves d and r}
   Let $d, r\in \ZZ$ and let $V$ be a $\VI^m$-module which is generated in degree $\leqslant d$ and related in degree $\leqslant r$. Then $\Sigma_i V$ and $\SSigma_i V$ are generated in degree $\leqslant d$ and related in degree $\leqslant r$. 
\end{lemma}

\begin{proof}
   Let 
   \[ Q \to P \to V \to 0 \]
   be an exact sequence where: 
   \begin{itemize}
      \item $P$ is a direct sum of $\VI^m$-modules of the form $M^{\VI^m}(\mathbf{n})$ with $|\mathbf{n}|\leqslant d$, 
      \item $Q$ is a direct sum of $\VI^m$-modules of the form $M^{\VI^m}(\mathbf{n})$ with $|\mathbf{n}|\leqslant r$. 
   \end{itemize}
   Since the functor $\Sigma_i$ is exact, we have an exact sequence 
   \[ \Sigma_i Q \to \Sigma_i P \to \Sigma_i V \to 0. \]
   By Proposition \ref{prop:shift_of_free}, we know that:
   \begin{itemize}
      \item $\Sigma_i P$ is a direct sum of $\VI^m$-modules of the form $M^{\VI^m}(\mathbf{n})$ with $|\mathbf{n}|\leqslant d$, 
      \item $\Sigma_i Q$ is a direct sum of $\VI^m$-modules of the form $M^{\VI^m}(\mathbf{n})$ with $|\mathbf{n}|\leqslant r$. 
   \end{itemize}
   Therefore $\Sigma_i V$ is generated in degree $\leqslant d$ and related in degree $\leqslant r$. Similarly for $\SSigma_i V$. 
\end{proof}

\section{The functors \(K_i\) and \(D_i\)} \label{sec: the functors Ki and Di}

\subsection{The functor $K_i$}
Let $i\in \{1, 2, \ldots, m\}$. 

For each $\mathbf{a}=(a_1, \ldots, a_m)\in \N^m$, define $\varpi_i\in \VI^m(\mathbf{a},\mathbf{a}+\mathbf{e}_i)$ as follows: 
\begin{itemize}
   \item the $j$-th component of $\varpi_i$ is $I_{a_j}$ for all $j\neq i$,
   \item the $i$-th component of $\varpi_i$ is the $(a_i+1)\times a_i$-matrix whose first row is $0$, and whose last $a_i$ rows form the matrix $I_{a_i}$.
\end{itemize}  

Let $V$ be a $\VI^m$-module. Then there is a natural homomorphism 
\[ \varepsilon_i: V\to \Sigma_i V \]
defined at each $\mathbf{a}\in \Ob(\VI^m)$ to be the map $(\varpi_i)_*: V_{\mathbf{a}} \to V_{\mathbf{a}+\mathbf{e}_i}$. Let 
\[ \xi_i: \Sigma_i V \to \SSigma_i V \] 
be the quotient map and let 
\[ \overline\varepsilon_i: V\to \SSigma_i V \]
be the composition of $\varepsilon_i$ and $\xi_i$.

Let $K_i V$ be the kernel of $\varepsilon_i: V\to \Sigma_i V$. 

\begin{lemma} \label{lem: kernel Ki}
   Let $V$ be any $\VI^m$-module. Then $K_i V$ is the kernel of $\overline\varepsilon_i: V\to \SSigma_i V$.
\end{lemma}

\begin{proof}
   Let $\mathbf{a}\in \N^m$ and $v\in V_{\mathbf{a}}$. We need to show that if $\xi_i((\varpi_i)_*(v))=0$, then $(\varpi_i)_*(v)=0$. 

   Since $|U_i(\mathbf{a})|$ is invertible in $\Bbbk$, we have a well-defined map 
   \[ \psi: (V_{\mathbf{a}+\mathbf{e}_i})_{U_i(\mathbf{a})} \to V_{\mathbf{a}+\mathbf{e}_i} \]
   which sends the class represented by $w\in V_{\mathbf{a}+\mathbf{e}_i}$ to the element 
   \[  |U_i(\mathbf{a})|^{-1}\sum_{g\in U_i(\mathbf{a})} g_*(w).  \]
   For every $g\in U_i(\mathbf{a})$, we have $g \varpi_i = \varpi_i$. Thus 
   \begin{align*}
      \psi(\xi_i((\varpi_i)_*(v))) &= |U_i(\mathbf{a})|^{-1}\sum_{g\in U_i(\mathbf{a})} g_*((\varpi_i)_*(v)) \\
      &= |U_i(\mathbf{a})|^{-1}\sum_{g\in U_i(\mathbf{a})} (\varpi_i)_*(v) \\ 
      &= (\varpi_i)_*(v).
   \end{align*}
   In particular, if $\xi_i((\varpi_i)_*(v))=0$, then $(\varpi_i)_*(v)=0$. 
\end{proof}

\subsection{The functor $D_i$}
Let $i\in \{1, 2, \ldots, m\}$. 

For any $\VI^m$-module $V$, let $D_i V$ be the cokernel of $\overline\varepsilon_i: V\to \overline\Sigma_i V$.
By Lemma \ref{lem: kernel Ki}, we have an exact sequence 
\[
   \xymatrix{
   0 \ar[r] & K_i V  \ar[r] & V \ar[r]^{\overline{\varepsilon}_i} & \SSigma_i V  \ar[r] & D_i V  \ar[r] & 0. } 
\]
This gives the following two short exact sequences:
\begin{gather}
   0 \longrightarrow K_i V \longrightarrow V \longrightarrow V/K_i V \longrightarrow 0,  \label{eq: ses 1} \\
   0 \longrightarrow V/K_i V \longrightarrow \SSigma_i V \longrightarrow D_i V \longrightarrow 0. \label{eq: ses 2}
\end{gather}
The short exact sequences \eqref{eq: ses 1} and \eqref{eq: ses 2} induce the long exact sequences
\begin{gather}
   \cdots \to \HH^{\VI^m}_{q+1} (V/K_i V) \to \HH^{\VI^m}_q(K_i V) \to \HH^{\VI^m}_q(V) \to \HH^{\VI^m}_q(V/K_i V) \to \cdots, \label{eq: les 1} \\
   \cdots \to \HH^{\VI^m}_{q+1} (D_i V) \to \HH^{\VI^m}_q(V/K_i V) \to \HH^{\VI^m}_q(\SSigma_i V) \to \HH^{\VI^m}_q(D_i V) \to \cdots. \label{eq: les 2}
\end{gather}

\begin{remark} \label{rem: exactness D}
   The functor $D_i$ is right exact.
\end{remark}

\begin{notation}
When $m=1$, we write 
\[ \varpi,\quad \varepsilon,\quad \overline{\varepsilon},\quad KV,\quad DV \] for 
\[ \varpi_1,\quad \varepsilon_1,\quad \overline{\varepsilon}_1,\quad K_1V,\quad D_1V, \]  
respectively. (The functor $D$ on $\VI\Mod$ is denoted by $\overline{\Delta}$ in \cite{na}.)  
\end{notation}

For any $\VI$-modules $V^{(1)}, \ldots, V^{(m)}$, we have:
\begin{equation}\label{eq: tensor_Di}
D_i\left( V^{(1)}\boxtimes \cdots\boxtimes V^{(m)} \right) 
= V^{(1)}\boxtimes \cdots \boxtimes \left( D V^{(i)} \right) \boxtimes \cdots\boxtimes V^{(m)}. 
\end{equation}

\begin{lemma} \label{lem: Di of free}
   Let $\mathbf{n}=(n_1, \ldots, n_m)\in \N^m$. 
   
   If $n_i=0$, then  
   \[ D_i M^{\VI^m}(\mathbf{n}) = 0. \]
   
   If $n_i\geqslant 1$, then 
   \[ D_i M^{\VI^m}(\mathbf{n}) \cong M^{\VI^m}(\mathbf{n}-\mathbf{e}_i)^{\oplus (q^{n_i}-1)}. \]
\end{lemma}

\begin{proof}
   By \eqref{eq: tensor} and \eqref{eq: tensor_Di}, it suffices to prove the lemma when $m=1$. 

   It is easy to see that the natural homomorphism $\overline{\varepsilon}: M^{\VI}(0)\to \SSigma M^{\VI}(0)$ is bijective and so 
   \[ DM^{\VI}(0) = 0. \]
   
   Let $n\in \N$ with $n\geqslant 1$. Recall that in the proof of Proposition \ref{prop:barshift1}, we defined for each row vector $\beta: \F^n \to \F$ a $\VI$-submodule $P(\beta)$ of $\Sigma M^{\VI}(n)$. The sum of $P(\beta)$ over all $\beta$ is a direct sum denoted by $P$ and we have an isomorphism $\eta: P \to \SSigma M^{\VI}(n)$ given by the composition of the inclusion map $P\to \Sigma M^{\VI}(n)$ and the quotient map $\Sigma M^{\VI}(n) \to \SSigma M^{\VI}(n)$. It is easy to see that the natural homomorphism $\varepsilon$ maps $M^{\VI}(n)$ bijectively onto $P(0)$, hence $D M^{\VI}(n)$ is isomorphic to the direct sum of $P(\beta)$ over all nonzero $\beta$. It follows  that 
   \[ DM^{\VI}(n) \cong M^{\VI}(n-1)^{\oplus (q^{n}-1)}.  \]   
\end{proof}

\begin{lemma} \label{lem:degD2}
   Let $d,r\in\ZZ$ and let $V$ be a $\VI^m$-module generated in degree $\leqslant d$ and related in degree $\leqslant r$. If $d\geqslant 0$, then $D_i V$ is a $\VI^m$-module generated in degree $\leqslant d-1$ and related in degree $\leqslant r$. 
\end{lemma}

\begin{proof}
   Let 
   \[ Q \to P \to V \to 0 \]
   be an exact sequence where: 
   \begin{itemize}
      \item $P$ is a direct sum of $\VI^m$-modules of the form $M^{\VI^m}(\mathbf{n})$ with $|\mathbf{n}|\leqslant d$, 
      \item $Q$ is a direct sum of $\VI^m$-modules of the form $M^{\VI^m}(\mathbf{n})$ with $|\mathbf{n}|\leqslant r$. 
   \end{itemize}
   Since the functor $D_i$ is right exact, we have an exact sequence 
   \[ D_i Q \to D_i P \to D_i V \to 0. \]
   By Lemma \ref{lem: Di of free}, we know that:
   \begin{itemize}
      \item $D_i P$ is a direct sum of $\VI^m$-modules of the form $M^{\VI^m}(\mathbf{n})$ with $|\mathbf{n}|\leqslant d-1$, 
      \item $D_i Q$ is a direct sum of $\VI^m$-modules of the form $M^{\VI^m}(\mathbf{n})$ with $|\mathbf{n}|\leqslant r$. 
   \end{itemize}
   Therefore $D_i V$ is generated in degree $\leqslant d-1$ and related in degree $\leqslant r$.
\end{proof}

\section{Shift theorem} \label{sec: shift theorem}

\subsection{The $m=1$ case}
Nagpal proved that if $V$ is a finitely generated $\VI$-module, then $\Sigma^s V$ and $\SSigma^s V$ are homology acyclic for all $s$ sufficiently large; see \cite{na}*{Proposition 3.10 and Theorem 4.38}. Our results in this section and the next depend on Nagpal's result. 

\begin{lemma} \label{lem: shift thm base case}
  Let $d, r\in \ZZ$. Let $V$ be a finitely generated $\VI$-module which is generated in degree $\leqslant d$ and related in degree $\leqslant r$. Then for any non-negative integer $s\geqslant d+r$, the $\VI$-modules $\Sigma^s V$ and $\SSigma^s V$ are homology acyclic. 
\end{lemma}

\begin{proof}
   For the natural shift functor, the lemma follows from \cite{gl_vi}*{Corollary 4.4} and \cite{na}*{Proposition 3.10}.

   To prove the lemma for the modified shift functor, we use induction on $d$. The base case $d=-1$ is trivial. Now let $d\geqslant 0$ and assume the following induction hypothesis: 
   \begin{equation*}
      \parbox{.85\textwidth}{\it For all finitely generated $\VI$-module $W$ which is generated in degree $\leqslant d-1$ and related in degree $\leqslant r$, if $s \geqslant d+r-1$, then $\SSigma^s W$ is homology acyclic.}
   \end{equation*}
   
   Let $s\geqslant d+r$. Recall that we have the exact sequence 
   \[ 0 \to KV \to V \to \SSigma V \to DV \to 0. \]
   Applying the functor $\SSigma^s$ to the above exact sequence, we obtain an exact sequence 
   \[ 0 \to \SSigma^s KV \to \SSigma^s V \to \SSigma^{s+1} V \to \SSigma^s DV \to 0. \]
   By \cite{gl_vi}*{Lemma 2.4 and Theorem 3.2}, we have $\deg KV \leqslant \max\{-1, d+r-1\}$ so $\SSigma^s KV=0$.  Therefore we have a short exact sequence 
   \[ 0 \to \SSigma^s V \to \SSigma^{s+1} V \to \SSigma^s DV \to 0. \]
   By Lemma \ref{lem:degD2} and our induction hypothesis, we have $\HH^{\VI}_i(\SSigma^s DV)=0$ for all $i\geqslant 1$. Therefore 
   \[ \HH^{\VI}_i(\SSigma^s V) \cong \HH^{\VI}_i(\SSigma^{s+1} V) \qquad \mbox{ for all }i\geqslant 1. \]
   Since $s$ is an arbitrary non-negative integer $\geqslant d+r$, it follows that 
   \[ \HH^{\VI}_i(\SSigma^s V) \cong \HH^{\VI}_i(\SSigma^n V) \qquad \mbox{ for all }n\geqslant s\mbox{ and }i\geqslant 1.\]  
   By \cite{na}*{Proposition 3.10 and Theorem 4.38}, there exists an integer $n\geqslant s$ such that $\HH^{\VI}_i(\SSigma^n V) = 0$ for all $i\geqslant 1$. We conclude that $\HH^{\VI}_i(\SSigma^s V)=0$ for all $i\geqslant 1$.
\end{proof}

\subsection{The $m\geqslant 2$ case}

\begin{lemma} \label{lem: hor and shift commute}
   Assume $m\geqslant 2$ and consider $\VI^m$ as the product category $\VI \times \VI^{m-1}$. Let $i\in\{2, \ldots, m\}$. Then $\HH_0^{\hor} \Sigma_i \cong \Sigma_i \HH_0^{\hor}$ and $\HH_0^{\hor} \SSigma_i \cong \SSigma_i \HH_0^{\hor}$.  
\end{lemma}

\begin{proof}
   We prove the lemma for the modified shift functors. The proof for the natural shift functors is similar.

   Let $\mathbf{n}=(n_1, \ldots, n_m)\in \N^m$. Let $V$ be any $\VI^m$-module. Then by definition,
   \[ (\HH_0^\hor (\SSigma_i V))_{\mathbf{n}} = (\SSigma_i V)_{\mathbf{n}}/ (\SSigma_i V)^\hor_{\mathbf{n}} . \]
   It follows that 
   \[ (\HH_0^\hor (\SSigma_i V))_{\mathbf{n}} \cong V_{\mathbf{n}+\mathbf{e}_i}/( V^\hor_{\mathbf{n}+\mathbf{e}_i}+C) \]
   where $C$ is the $\Bbbk$-submodule of $V_{\mathbf{n}+\mathbf{e}_i}$ generated by all elements of the form $g_*(v)-v$ for some $g\in U_i(\mathbf{n})$ and $v\in V_{\mathbf{n}+\mathbf{e}_i}$.
   
   On the other hand, we also have
   \begin{align*}
      (\SSigma_i \HH_0^\hor(V))_{\mathbf{n}} 
      &= ((\HH_0^\hor(V))_{\mathbf{n}+\mathbf{e}_i})_{U_i(\mathbf{n})}  \\
      &= ( V_{\mathbf{n}+\mathbf{e}_i} / V^\hor_{\mathbf{n}+\mathbf{e}_i} )_{U_i(\mathbf{n})}  \\
      &\cong V_{\mathbf{n}+\mathbf{e}_i}/( V^\hor_{\mathbf{n}+\mathbf{e}_i} + C).
   \end{align*}   
\end{proof}

\begin{theorem}[Shift theorem] \label{thm: shift}
      Let $m,d,r\in\mathbb{Z}$ with $m\geqslant 1$ and $d, r\geqslant -1$. Let $V$ be a finitely generated $\VI^m$-module which is generated in degree $\leqslant d$ and related in degree $\leqslant r$. Then for any non-negative integer $s\geqslant d+r$, the $\VI^m$-modules $\Sigma_1^s \Sigma_2^s \cdots \Sigma_m^s V$ and $\SSigma_1^s \SSigma_2^s \cdots \SSigma_m^s V$ are homology acyclic. 
\end{theorem}

\begin{proof}
   We prove the theorem for the modified shift functors. The proof for the natural shift functors is similar.

   We use induction on $m$, the base case $m=1$ being Lemma \ref{lem: shift thm base case}. Now let $m\geqslant 2$ and assume that the theorem holds for finitely generated $\VI^{m-1}$-modules. We shall consider $\VI^m$ as the product category $\VI \times \VI^{m-1}$.
   
   Let $s\geqslant d+r$. For all $(n_2, \ldots, n_m)\in \N^{m-1}$, we have
   \begin{equation} \label{eq:shifts and restrict}
      (\SSigma_1^s \SSigma_2^s \cdots \SSigma_m^s V)_{(-, n_2, \ldots, n_m)} = \SSigma^s ((\SSigma_2^s \cdots \SSigma_m^s V)_{(-, n_2, \ldots, n_m)}). 
   \end{equation}
      By Proposition \ref{prop:barshift2}, the $\VI^m$-module $\SSigma_2^s \cdots \SSigma_m^s V$ is generated in degree $\leqslant d$ and related in degree $\leqslant r$. 
   Hence by \cite{gta}*{Lemma 2.3}, the $\VI$-module 
   \[ (\SSigma_2^s \cdots \SSigma_m^s V)_{(-, n_2, \ldots, n_m)} \] 
   is also generated in degree $\leqslant d$ and related in degree $\leqslant r$. It follows by Lemma \ref{lem: shift thm base case} and \eqref{eq:shifts and restrict} that $(\SSigma_1^s \SSigma_2^s \cdots \SSigma_m^s V)_{(-, n_2, \ldots, n_m)}$ is a $\VI$-module which is homology acyclic. 
 
   Let $p\geqslant 1$. Then
   \begin{align*}
      \HH_p^{\VI^m}( \SSigma_1^s \SSigma_2^s \cdots \SSigma_m^s V ) 
      \cong & \HH_p^\ver \HH_0^\hor ( \SSigma_1^s \SSigma_2^s \cdots \SSigma_m^s V ) & \mbox{by Lemma \ref{lem: spec seq degenerate}}\\
      \cong & \HH_p^\ver \HH_0^\hor ( \SSigma_2^s \cdots \SSigma_m^s \SSigma_1^s V ) & \mbox{by Lemma \ref{lem:modified shifts commute}}\\
      \cong & \HH_p^\ver  ( \SSigma_2^s \cdots \SSigma_m^s \HH_0^\hor(\SSigma_1^s V) ) & \mbox{by Lemma \ref{lem: hor and shift commute}}. 
   \end{align*}
   Hence by \eqref{eq:hor and ver homology}, it suffices to show that for each $n_1\in \N$, the $\VI^{m-1}$-module 
   \[ ( \SSigma_2^s \cdots \SSigma_m^s \HH_0^\hor(\SSigma_1^s V) )_{(n_1, -)} \]
   is homology acyclic. We have 
   \[  ( \SSigma_2^s \cdots \SSigma_m^s \HH_0^\hor(\SSigma_1^s V) )_{(n_1, -)} 
      = \SSigma_1^s \cdots \SSigma_{m-1}^s ((\HH_0^\hor(\SSigma_1^s V) )_{(n_1, -)}).  \]
   Therefore by inductive hypothesis, it suffices to prove that the $\VI^{m-1}$-module $(\HH_0^\hor(\SSigma_1^s V) )_{(n_1, -)}$ is generated in degree $\leqslant d$ and related in degree $\leqslant r$. 

   By Proposition \ref{prop:barshift2}, we know that $\SSigma_1^s V$ is generated in degree $\leqslant d$ and related in degree $\leqslant r$. Hence by Lemma \ref{lem: hor prin proj res}, the $\VI^{m-1}$-module $(\HH_0^\hor(\SSigma_1^s V) )_{(n_1, -)}$ is generated in degree $\leqslant d$ and related in degree $\leqslant r$.
\end{proof}

\subsection{The $\FI^m$-module analog}
Li and Yu proved that if $V$ is a finitely generated $\FI^m$-module over a noetherian commutative ring, then for all $s$ sufficiently large, the $\FI^m$-module $\Sigma_1^s \cdots \Sigma_m^s V$ is homology acyclic; see  \cite{liyu}*{Proposition 4.4 and Proposition 4.10}. Their result can be strengthened to the following.

\begin{theorem} \label{thm: shift thm for FIm}
   Let $m,d,r\in\mathbb{Z}$ with $m\geqslant 1$ and $d, r\geqslant -1$. Let $V$ be an $\FI^m$-module over a commutative ring $k$. Assume that $V$ is generated in degree $\leqslant d$ and related in degree $\leqslant r$. Then for any non-negative integer $s\geqslant d+r$, the $\FI^m$-module $\Sigma_1^s \cdots \Sigma_m^s V$ is homology acyclic.
\end{theorem}

We do not require $V$ to be finitely generated or $k$ to be noetherian in Theorem \ref{thm: shift thm for FIm}.

We omit the proof of Theorem \ref{thm: shift thm for FIm} because it is an essentially verbatim analog of the proof of Theorem \ref{thm: shift}, noting that in place of Lemma \ref{lem: shift thm base case}, we use the $m=1$ case which is known for example from \cite{gan}*{Theorem 12}.

\section{Upper bound on regularity} \label{sec: upper bound}

\subsection{The function $\rho_m$}
We recall from \cite{gta}*{Definition 1.5} the function 
\[ \rho_m: \ZZ \times \ZZ \to \ZZ. \] 

\begin{definition} \label{def:rho}
   Let $m,d,r\in\mathbb{Z}$ with $m\geqslant 1$ and $d, r\geqslant -1$.

   If $m=1$, then let
   \[ \rho_m(d,r) = \max\{d, d+r-1\}. \]

   If $m\geqslant 2$ and $d=-1$, then let 
   \[ \rho_m(d,r) = -1. \]
   
   If $m\geqslant 2$ and $d\geqslant 0$, then let 
   \[ \rho_m(d,r) = \max\{ \rho_{m-1}(\rho'_m(d,r), \rho''_m(d,r)), 1+\rho_m(d-1,r) \} \]
   where 
   \begin{align*}
      \rho'_m(d,r) &= \max\{ 2+\rho_m(d-1, r), r \}  ,\\
      \rho''_m(d,r) &=  \max\{ 3+\rho_m(d-1, r), 4+ \rho_1(d,r) + \rho_{m-1}(d,r) \}.
   \end{align*}
\end{definition}

Later, we shall make use of the following properties of the functions $\rho'_m$ and $\rho''_m$.

\begin{lemma} \label{lem: rho' and rho'' inequalities}
   Let $m, d, r \in \Z$ with $m\geqslant 2$, $d\geqslant 0$, $r\geqslant -1$. Then we have:
   \begin{enumerate}[(a)]
      \item $\rho'_m(d,r) > d$.
      \item $\rho''_m(d,r) > r$.
   \end{enumerate}
\end{lemma}

\begin{proof}
   (a) We have: 
   \begin{align*}
      \rho'_m(d,r) &\geqslant 2 + \rho_m(d-1, r) \\
      &\geqslant 2 + (d-1) \quad\quad\quad \mbox{by \cite{gta}*{Corollary 4.4}} \\
      &> d.
   \end{align*}

   (b) We have: 
   \begin{align*}
      \rho''_m(d,r) &\geqslant 4+ \rho_1(d,r) + \rho_{m-1}(d,r) \\
      &\geqslant 4+(d+r-1)+(-1) \\
      &> r.
   \end{align*}
\end{proof}

\subsection{Main result}
Our main result is the following. 

\begin{theorem} \label{thm: regularity}
   Let $m, d, r\in \Z$ with $m\geqslant 1$ and $d,r\geqslant -1$. Let $V$ be a finitely generated $\VI^m$-module which is generated in degree $\leqslant d$ and related in degree $\leqslant r$. Then
   \[ \reg(V) \leqslant \rho_m(d,r). \]
\end{theorem}

The $\FI^m$-module analog of Theorem \ref{thm: regularity} was proved in \cite{gta}*{Theorem 1.6} without assuming that $V$ is finitely generated or $\Bbbk$ is noetherian. In the following remark, we discuss where the method of proof of \cite{gta}*{Theorem 1.6} fails for $\VI^m$-modules. 

\begin{remark} \label{rem: pf in gta}
   \begin{enumerate}[(a)]
      \item We write $\Sigma_1, \ldots, \Sigma_m$ for the shift functors on $\FI^m\Mod$. For any $i\in \{1, \ldots, m\}$ and $\FI^m$-module $V$, there is a natural homomorphism $V\to \Sigma_i V$ whose cokernel $D_i V$ satisfies 
      \[ t_0(D_i V)\leqslant \max\{-1, t_0(V)-1\}. \] 
      Similarly for $\VI^m$-modules with $\SSigma_i$ in place of $\Sigma_i$. This gives a way to prove results on (certain classes of) $\FI^m$-modules or $\VI^m$-modules $V$ by induction on $t_0(V)$. 

      \item We recall from \cite{gta}*{Lemma 3.4} that for any $\FI^m$-module $V$, we have 
      \[ \deg(V) \leqslant 1 + \max\left\{\deg(\Sigma_i V) : i=1, \ldots, m\right\}. \]  
      This may be \emph{false} for a $\VI^m$-module $V$ with $\SSigma_i$ in place of $\Sigma_i$ (even when $m=1$). This lemma is used in Step 4 in the proof of \cite{gta}*{Theorem 1.6} in an essential way. It is also used in the proof of \cite{gta}*{Corollary 3.11} which says that for any $\FI^m$-module $V$ with $\deg(V)<\infty$, we have 
      \[ \reg(V)\leqslant \deg(V). \] 
      However, only the $m=1$ case of \cite{gta}*{Corollary 3.11} is actually used later (in the proof of \cite{gta}*{Proposition 4.2}), and this $m=1$ case for finitely generated $\VI$-modules has been proved by Nagpal \cite{na}*{Lemma 5.10}. 

      \item The method of Gan and Li in \cite{gl_product} for proving finiteness of regularity is not applicable to $\VI^m$-modules.
   \end{enumerate}
\end{remark}

We now prove our main result. We shall refer the reader to \cite{gta} when the arguments for $\VI^m$-modules are essentially the same as the arguments for $\FI^m$-modules.

\begin{proof}[Proof of Theorem \ref{thm: regularity}]
   For any integers $m\geqslant 1$ and $d\geqslant -1$, denote by $\mathfrak{T}(m, d)$ the following statement:
   \begin{equation*}
      \parbox{.85\textwidth}{\it For any integer $r\geqslant -1$, if $V$ is a finitely generated $\VI^m$-module which is generated in degree $\leqslant d$ and related in degree $\leqslant r$, then $\reg(V)\leqslant \rho_m(d,r)$.}
   \end{equation*}
   We shall prove that $\mathfrak{T}(m,d)$ is true by nested induction. The outer induction is over $m$ and the inner induction is over $d$. 

   By \cite{gl_vi}*{Theorem 3.2}, we know that $\mathfrak{T}(1,d)$ is true for all $d\geqslant -1$. It is easy to see that $\mathfrak{T}(m,-1)$ is true for all $m\geqslant 1$. 

   Now fix $m\geqslant 2$ and $d\geqslant 0$. Assume that:
   \begin{itemize}
      \item $\mathfrak{T}(m-1, c)$ is true for all $c\geqslant -1$;
      \item $\mathfrak{T}(m, d-1)$ is true.
   \end{itemize}
   To prove that $\mathfrak{T}(m,d)$ is true, fix an integer $r\geqslant -1$ and let $V$ be any finitely generated $\VI^m$-module which is generated in degree $\leqslant d$ and related in degree $\leqslant r$. We want to show that 
   \[ \reg(V) \leqslant \rho_m(d,r). \] 
   We do this in several steps. In the following, let $i\in \{1, \ldots, m\}$.
 
   \bigskip 
   
   {\bf Step 1. Bound $t_2(\SSigma_i V)$.}
   
   \medskip 

   By Lemma \ref{lem: shift preserves d and r}, we know that $\SSigma_i V$ is generated in degree $\leqslant d$ and related in degree $\leqslant r$. 

   Let $(n_1, n_2, \ldots, n_m)\in \N^m$. We deduce by \cite{gta}*{Lemma 2.3} that $\SSigma_i V_{(-, n_2, \ldots, n_m)}$ and $\SSigma_i V_{(n_1, -)}$ are generated in degree $\leqslant d$ and related in degree $\leqslant r$. Hence by $\mathfrak{T}(1,d)$ and $\mathfrak{T}(m-1,d)$, we have: 
   \begin{align*}
      \reg\left( \SSigma_i V_{(-,n_2,\ldots, n_m)} \right) &\leqslant \rho_1 (d,r),\\
      \reg\left( \SSigma_i V_{(n_1,-)}\right) &\leqslant \rho_{m-1} (d,r).
   \end{align*}
   Applying Proposition \ref{prop: bound ti by 2i} to $\SSigma_i V$ gives:
   \begin{equation} \label{eq:t2bound}
      t_2(\SSigma_i V) \leqslant 4 + \rho_1(d,r) + \rho_{m-1}(d,r).
   \end{equation}

   \bigskip 

   {\bf Step 2. Bound $t_0(K_i V)$ and $t_1(K_i V)$.}

   \medskip 
   
   Let $q\in \N$. Then we have:
   \begin{align*}
      t_q(K_i V)  &\leqslant \max\{ t_q(V), t_{q+1}(V/K_i V) \} & \mbox{by \eqref{eq: ses 1}} \\
      &\leqslant \max\{ t_q(V), t_{q+1}(\SSigma_i V), t_{q+2}(D_i V) \} & \mbox{by \eqref{eq: les 2}}. 
   \end{align*}

   When $q=0$, we have:
   \[  t_0(K_i V) \leqslant \max\{ t_0(V), t_1(\SSigma_i V), t_2(D_i V) \}. \]
   By Lemma \ref{lem: shift preserves d and r}, Lemma \ref{lem:degD2}, and $\mathfrak{T}(m, d-1)$, we get:
   \[ t_0(K_i V)  \leqslant \max\{ d, r, 2+\rho_m(d-1, r) \}. \]
   Hence by Definition \ref{def:rho} and Lemma \ref{lem: rho' and rho'' inequalities}:
   \begin{equation} \label{eq: step 2 bound t0KV}
      t_0(K_i V)  \leqslant \rho'_m(d, r). 
   \end{equation}

   When $q=1$, we have:
   \[ t_1(K_i V) \leqslant \max\{ t_1(V), t_2(\SSigma_i V), t_3(D_i V) \}. \]
   By \eqref{eq:t2bound} from Step 1, Lemma \ref{lem:degD2}, and $\mathfrak{T}(m, d-1)$, we get:
   \[  t_1(K_i V) \leqslant  \max\{ r, 4+\rho_1(d,r)+\rho_{m-1}(d,r), 3+\rho_m(d-1, r) \}. \]
   Hence by Definition \ref{def:rho} and Lemma \ref{lem: rho' and rho'' inequalities}:
   \begin{equation} \label{eq: step 2 bound t1KV}
      t_1(K_i V) \leqslant \rho''_m(d, r).
   \end{equation}

   \bigskip 

   {\bf Step 3. Bound $\reg(K_i V)$.}

   \medskip 

   Similarly to Step 3 in the proof of \cite{gta}*{Theorem 1.6}, we have: 
     \begin{equation} \label{eq:bound reg of KV}
      \reg(K_i V) \leqslant \rho_{m-1}(\rho'_m(d,r), \rho''_m(d,r)). 
   \end{equation}

   The proof for finitely generated $\VI^m$-modules uses:
   \begin{itemize}
      \item \eqref{eq: step 2 bound t0KV} and \eqref{eq: step 2 bound t1KV} from Step 2;
      \item the $\VI^m$-analog of \cite{gta}*{Lemma 3.1}, which is proved in the same way with $\VI$ in place of $\FI$;
      \item the $\VI^m$-analog of \cite{gta}*{Proposition 4.2}, which is proved in the same way with $\VI$ in place of $\FI$, and with \cite{na}*{Lemma 5.10} in place of \cite{gta}*{Corollary 3.11}.
   \end{itemize}

   \bigskip 

   {\bf Step 4. Bound $\reg(V)$.}

   \medskip 

   We remind the reader that \eqref{eq:bound reg of KV} in Step 3 is applicable to any $\VI^m$-module $V$ which is generated in degree $\leqslant d$ and related in degree $\leqslant r$.

   We claim the following:
   \begin{equation*}
      \parbox{.85\textwidth}{\it Let $W$ be a finitely generated $\VI^m$-module which is generated in degree $\leqslant d$ and related in degree $\leqslant r$.  
      Assume that $\reg(\SSigma_i W) \leqslant \rho_m(d,r)$. Then $\reg(W) \leqslant \rho_m(d,r)$.}
   \end{equation*}

   To prove the claim, let $q\in \N$. Then: 
   \begin{align*}
     t_q(W) &\leqslant \max\{ t_q(K_i W), t_q(W/K_i W) \}  & \mbox{by \eqref{eq: les 1}}\\
     &\leqslant \max\{ t_q(K_i W), t_q(\SSigma_i W), t_{q+1}(D_i W) \} & \mbox{by \eqref{eq: les 2}}\\
     &\leqslant  q + \max\{ \reg(K_i W), \reg(\SSigma_i W), 1+\reg(D_i W) \}. &
   \end{align*}
   By \eqref{eq:bound reg of KV} from Step 3, Lemma \ref{lem:degD2}, and $\mathfrak{T}(m,d-1)$, it follows that: 
   \[ t_q(W) \leqslant  q + \max\{ \rho_{m-1}(\rho'_m(d,r), \rho''_m(d,r)), \rho_m(d,r), 1+\rho_m(d-1, r) \}. \]
   Hence by Definition \ref{def:rho}:
   \[ t_q(W) \leqslant q + \rho_m(d,r). \]
   This proves the claim.

   Now let $s=\max\{0, d+r\}$. Then
   \begin{align*}
      \reg(\SSigma^s_1 \SSigma^s_2 \cdots \SSigma^s_m V) &= t_0( \SSigma^s_1 \SSigma^s_2 \cdots \SSigma^s_m V ) & \mbox{by Theorem \ref{thm: shift}}\\
      &\leqslant d & \mbox{by Lemma \ref{lem: shift preserves d and r}} \\
      &\leqslant \rho_m(d,r) & \mbox{by \cite{gta}*{Corollary 4.4}}.
   \end{align*} 
   Keeping Lemma \ref{lem: shift preserves d and r} in mind, we can therefore apply the above claim successively to see that:
   \begin{align*}
      & \quad \reg(\SSigma^s_1 \SSigma^s_2 \cdots \SSigma^s_m V) \leqslant \rho_m(d,r) \\
      \Longrightarrow & \quad \reg(\SSigma^{s-1}_1 \SSigma^s_2 \cdots \SSigma^s_m V) \leqslant \rho_m(d,r) \\ 
      \Longrightarrow & \quad \reg(\SSigma^{s-2}_1 \SSigma^s_2 \cdots \SSigma^s_m V) \leqslant \rho_m(d,r) \\ 
      & \qquad\qquad \vdots \\
      \Longrightarrow & \quad \reg(\SSigma^2_m V) \leqslant \rho_m(d,r) \\
      \Longrightarrow & \quad \reg(\SSigma_m V) \leqslant \rho_m(d,r) \\
      \Longrightarrow & \quad \reg(V) \leqslant \rho_m(d,r).
   \end{align*}
   
   This completes the proof of the theorem. 
\end{proof}

\end{document}